\documentclass[a4paper]{amsart}

\usepackage{amssymb}

\usepackage{graphicx}

\usepackage{bbm}

\newcommand\xx{\mathbf{x}}
\newcommand\kbar{{\overline{k}}}
\newcommand\Dfive{\mathbf{D}_5}
\newcommand\Esix{\mathbf{E}_6}
\newcommand\Eseven{\mathbf{E}_7}
\newcommand\Eeight{\mathbf{E}_8}
\newcommand\Eff{\Lambda_{\rm eff}}
\newcommand\Nef{\Lambda_{\rm eff}^\vee}
\newcommand\tS{{\widetilde S}}

\newcommand{\bP}{{\mathbbm P}}
\DeclareMathOperator\SSS{S}
\DeclareMathOperator\Gal{Gal}
\DeclareMathOperator\Pic{Pic}
\DeclareMathOperator\vol{vol}
\DeclareMathOperator\rk{rk}

\newcommand{\bbQ}{{\mathbbm Q}}
\newcommand{\bbR}{{\mathbbm R}}
\newcommand{\bbZ}{{\mathbbm Z}}

\newcommand{\br}{ }
\newcommand{\brr}{, }

\newcounter{abc}
\newenvironment{abc}{\begin{list}{\rm \alph{abc}) }{\usecounter{abc} \leftmargin=0.0pt \labelsep=0.0pt \listparindent=0.0pt \labelwidth=0.0pt \parsep=\smallskipamount \itemsep=0.0pt \topsep=0.0pt \partopsep=\smallskipamount}}{\end{list}}
\newcounter{iii}
\newenvironment{iii}{\begin{list}{\rm \roman{iii}) }{\usecounter{iii} \leftmargin=0.0pt \labelsep=0.0pt \listparindent=0.0pt \labelwidth=0.0pt \parsep=\smallskipamount \itemsep=0.0pt \topsep=0.0pt \partopsep=\smallskipamount}}{\end{list}}
\newcounter{III}
\newenvironment{III}{\begin{list}{\rm \Roman{III}) }{\usecounter{III} \leftmargin=0.0pt \labelsep=0.0pt \listparindent=0.0pt \labelwidth=0.0pt \parsep=\smallskipamount \itemsep=0.0pt \topsep=0.0pt \partopsep=\smallskipamount}}{\end{list}}

\newtheorem{theorem}{Theorem}[section]

\theoremstyle{definition}
\newtheorem{definition}[theorem]{Definition}
\newtheorem{example}[theorem]{Example}
\newtheorem{algo}[theorem]{Algorithm}

\theoremstyle{remark}
\newtheorem{ttt}[theorem]{}
\newtheorem{remark}[theorem]{Remark}
\newtheorem{remarks}[theorem]{Remarks}

\numberwithin{equation}{section}

\begin{document}

\title{On the factor alpha in Peyre's constant}

\author{Ulrich Derenthal}

\address{Mathematisches Institut, \!Ludwig-Maximilians-Universit\"at M\"unchen, 
  \!Theresienstr.\,39, D-80333 M\"un\-chen, Germany}

\email{ulrich.derenthal@mathematik.uni-muenchen.de}

\thanks{The first author was partly supported by Deutsche
  Forschungsgemeinschaft (DFG) grant DE~1646/2-1, SNF grant 200021\_124737/1,
  and by the Center for Advanced Studies of LMU M\"unchen.}

\author{Andreas-Stephan Elsenhans}

\address{School of Mathematics and Statistics F07,
  University of Sydney,
  NSW 2006,
  Sydney,
  Australia}

\email{stephan@maths.usyd.edu.au}

\thanks{The second author was supported in part by the Deutsche Forschungsgemeinschaft (DFG) through a funded research~project.}

\author{J\"org Jahnel}

\address{D\'epartement Mathematik, Universit\"at Siegen, Walter-Flex-Str.~3,
  D-57068 Siegen, Germany}

\email{jahnel@mathematik.uni-siegen.de}


\date{}

\dedicatory{}

\keywords{Peyre's constant, del Pezzo surface, polyhedron, volume, {\tt polymake}}

\begin{abstract}
  For an arbitrary del Pezzo surface $S$, we compute $\alpha(S)$, which is the
  volume of a certain polytope in the dual of the effective cone of~$S$, using
  \texttt{magma} and \texttt{polymake}. The constant $\alpha(S)$ appears in
  Peyre's conjecture for the leading term in the asymptotic formula for
  the number of rational points of bounded height on $S$ over number fields.
\end{abstract}

\maketitle

\section{Introduction}

Let $S$ be a smooth del Pezzo surface defined over a number field $k$;
for simplicity, we assume in this introduction that its degree is $\le
\!7$. If $S$ contains infinitely many $k$-rational points, Manin's
conjecture \cite{FMT} makes a precise prediction of the number of
\mbox{$k$-rational} points of bounded height on $S$: Let $U$ be the
complement of the lines (more precisely, $(-1)$-curves; see
Section~\ref{sec:del_pezzo} for details) on $S$, and let $H$ be an
anticanonical height function on the set $S(k)$ of $k$-rational
points. Then the number
\begin{equation*}
  N_{U,H}(B) = \#\{\xx \in U(k) \mid H(\xx) \le B\}
\end{equation*}
of $k$-rational points of height at most $B$ outside the lines is
expected to grow asymptotically as $c_{S,H} B(\log B)^{\varrho(S)-1}$ for
$B \to \infty$, where $\varrho(S)$ is the rank of the Picard group
$\Pic(S)$ over $k$, and $c_{S,H} > 0$ is a constant whose value was
predicted by Peyre \cite{Pe} to be
\begin{equation*}
  c_{S,H} = \alpha(S)\beta(S)\omega(S,H).
\end{equation*}
See \cite{Bro} for an overview of this conjecture for del Pezzo
surfaces, and, e.g.,~\cite{EJ3} for experimental results.  Here,
$\omega(S,H)$ is essentially a product of local densities that shows
up similarly for example in applications of the Hardy-Littlewood
circle method, $\beta(S)$ is defined as
$\#H^1(\Gal(\kbar/k),\Pic(S_\kbar))$, and $\alpha(S)$ is the volume of
a~certain polytope in $\smash{\Pic(S)_\bbR = \Pic(S) \otimes_\bbZ \bbR \cong
\bbR^{\varrho(S)}}$.

In this note, we are interested in $\alpha(S)$. For any smooth del Pezzo
surface $S$, it can be defined as follows; see also
\cite[D\'efinition~2.4]{Pe} and \cite[Definition~2.4.6]{BT}.

\begin{definition}
  Let $\Eff(S)$ be the effective cone of $S$ in $\Pic(S)_\bbR$, let $\Nef(S)$
  be its dual cone with respect to the intersection form, and let $(-K_S) \in
  \Pic(S)$ be the anticanonical class of $S$. Then
  \begin{equation*}
    \alpha(S) = \varrho(S) \cdot \vol\{x \in \Nef(S) \mid \langle x, -K_S \rangle \leq 1\},
  \end{equation*}
  where the volume on $\Pic(S)_\bbR$ is normalized such that the lattice
  $\Pic(S)$ has covolume~$1$.
\end{definition}

\noindent
Over an algebraic closure $\kbar$ of $k$, the effective cone $\Eff(S_\kbar)$
is generated by the lines on $S_\kbar$. We will see that $\alpha(S)$ depends
only on the degree $d$ of $S$ and on the combinatorial structure of the action
of $\Gal(\kbar/k)$ on the lines on $S_\kbar$ via a subgroup $W(S)$ of their
finite symmetry group $W$, which is a Weyl group.

For smooth \emph{split\/} del Pezzo surfaces ($\#W(S)=1$, i.e., when
each line is defined over $k$), a formula for
$\alpha(S)$ was found in \cite[Theorem~4]{De}. For smooth
\emph{non-split} del Pezzo surfaces of degree $\ge \!5$, the values of
$\alpha(S)$ were determined in \cite[Section~7B]{DJT}.

\begin{remark}\label{rem:singular}
For a del Pezzo surfaces $S'$ with $\textbf{ADE}$ singularities, Manin's
conjecture can be formulated similarly. In this case, the expected constants
must be computed on the minimal desingularization $\tS'$ of $S'$. We have
$\alpha(\tS') = \frac{\alpha(S)}{\#W'}$, where $S$ is a smooth del Pezzo
surface of the same degree as $S'$ with a corresponding action of
$\Gal(\kbar/k)$ on its lines over $\kbar$ and $W'$ is a Weyl group associated
to the singularities of $S'$. See \cite[Theorem~1.3, Corollary~7.5]{DJT}
for details.
\end{remark}

\noindent
It remains to determine $\alpha(S)$ for smooth non-split del Pezzo
surfaces of degree $d \le 4$. By Remark~\ref{rem:singular}, this is also
relevent for the singular case. Unfortunately, it seems that this cannot be
done without certain case-by-case considerations. We~will give a general
algorithm and then apply it to the cases $d =1, \dots, 4$.

By \cite[Section~7A]{DJT}, we are reduced to finitely many cases
corresponding to the conjugacy classes of subgroups of the Weyl
groups~$W$.  For~example, when $d=1$, the number of cases is $62092$;
for cubic surfaces, there are $350$~cases.

Theorem~\ref{inkl} is our key result that allows us to reduce the
number of cases to an order of magnitude that could in principle be
treated by hand without too much effort: If $S$ and $S'$ have the same
degree, Picard groups of the same rank and if $W(S)$ is contained in a
conjugate of $W(S')$, then $\alpha(S)=\alpha(S')$.  Using this, only
$14$ cases are left for $d=4$, only $17$ for $d=3$, only $32$ for
$d=2$, and $41$ for~$d=1$.

Nevertheless, we choose to treat these cases not by hand, but with the
help of the software {\tt polymake}~\cite{GJ} that allows to compute
volumes of polytopes. The~subgroups of $W(R_d)$ can be obtained using
{\tt gap}~\cite{gap} or {\tt magma}~\cite{BCFS}. Our results are
summarized in Tables~\ref{tab:degfour}, \ref{Tab}
and~\ref{tab:degtwoone}.

At~least for $d \ge 3$, it is possible to extract the value of
$\alpha(S)$ for a concretely given del Pezzo surface from
these~tables. See Remark~\ref{rem:orbits} for a detailed~discussion.

\section{Del Pezzo surfaces}\label{sec:del_pezzo}

We recall some facts on the structure of del Pezzo surfaces. See \cite{Ma} for
more~details. A (smooth / ordinary) del Pezzo surface $S$ over a field $k$ is
a smooth projective variety of dimension $2$ defined over $k$ whose
anticanonical class $(-K_S)$ is~ample. Its degree $d$ is the self-intersection
number $(-K_S,-K_S)$ of the anticanonical~class. Over $\kbar$, it is
isomorphic to $\bP^2$ (of degree $9$), $\bP^1 \times \bP^1$ (of degree $8$) or
the blow-up of $\bP^2$ in $r$ points in general position (with $r \in \{1,
\dots, 8\}$, of degree $d=9-r$). Its geometric Picard group $\Pic(S_\kbar)$ is
free of rank $10-d$.

A $(-1)$-curve is a curve $E$ on $S_\kbar$ such that its class $[E] \in
\Pic(S_\kbar)$ satisfies $(-K_S,[E])=1$ and $([E],[E])=-1$.  A~$(-2)$-class is
an element $L$ of $\Pic(S_\kbar)$ with $(-K_S,L)=0$ and $(L,L)=-2$.

\begin{table}[ht]
  \centering
  \[\begin{array}[h]{|c||c|c|c|c|}
    \hline
    d & 4 & 3 & 2 & 1 \\
    \hline\hline
    R_d & \Dfive & \Esix & \Eseven & \Eeight \\
    \#W(R_d) & 1920 & 51840 & 2903040 & 696729600 \\
    \hline
    w_d & 1328 & 1161 & - & -\\
    c_d & 197 & 350 & 8074 & 62092\\
    c_d' & 38 & 91 & 1071 & 13975 \\
    c_d'' & 14 & 17 & 32 & 41 \\
    \hline
    N_d & 16 & 27 & 56 & 240 \\
    N_{d,0} & 10 & 16 & 27 & 56 \\
    N_{d,1} & 5 & 10 & 27 & 126 \\
    N_{d,2} & 0 & 0 & 1 & 56 \\
    N_{d,3} & 0 & 0 & 0 & 1 \\
    \hline
  \end{array}\]\vskip1mm
  \caption{The root systems associated to del Pezzo surfaces  of degree
   $\leq\! 4$}
  \label{tab:weyl_groups}\vskip-5mm
\end{table}

From here, we restrict to the case $d \le 7$. Then $\Pic(S_\kbar)$ and its
intersection form depend only on the degree $d$ of $S$
(cf.~\cite[Section~3]{DJT}). Consequently the same holds for the
$(-2)$-classes (which form a root system $R_d$ whose type can be found in
Table~\ref{tab:weyl_groups} for $d\le 4$) and the classes of the $(-1)$-curves
(whose number is denoted by $N_d$) with their pairwise intersection numbers.

The geometric Picard group $\Pic(S_\kbar)$ is generated by the classes of the
$(-1)$-curves. The symmetry group of $\Pic(S_\kbar)$ respecting the
intersection form is the Weyl group~$W(R_d)$ associated to the root system
$R_d$ of the $(-2)$-classes in $\Pic(S_\kbar)$. Via its action on the set of
$(-1)$-curves, $W(R_d)$ can be regarded as a subgroup of the symmetric group
$\SSS_{N_d}$. Over~$\kbar$, the effective cone $\Eff(S_\kbar) \subset
\Pic(S_\kbar)_\bbR$ is generated by the $(-1)$-curves (see
\cite[Proposition~3.9]{DJT}, for example). All this depends only on the degree
$d \le 7$.

\begin{remark}\label{rem:matrix}
  For $d \le 5$, the group $W(R_d)$ acts transitively on the sets of pairs of
  $(-1)$-curves with fixed intersection number (in the set $\{-1, \dots,
  3\}$). This~gives a way to recover $\Pic(S_\kbar)$ with the intersection
  form and the classes of the $(-1)$-curves if $W(R_d)$ is given as a subgroup
  of $\SSS_{N_d}$. Namely, any $(-1)$-curve $E$ has intersection number $i$
  with precisely $N_{d,i}$ other $(-1)$-curves, as listed in
  Table~\ref{tab:weyl_groups}. Therefore, the pairs of $(-1)$-curves with
  intersection number $i$ form an orbit with precisely $\smash{\frac{N_d \cdot
    N_{d,i}}2}$ elements under the action of $W(R_d)$. This gives a $N_d \times
  N_d$ intersection matrix $M_d$ of rank $10-d$, with $\Pic(S_\kbar) \cong
  \bbZ^{N_d} / \ker M_d$.

  For $d \in \{1,2\}$, we have $N_{d,0}=N_{d,3-d}$, giving two orbits of the
  same size. This results in two candidates for $M_d$, where the correct one
  can be identified by the expected~rank.
\end{remark}

\noindent
However, $\Pic(S)$ and $\alpha(S)$ depend fundamentally on the
structure of $S$ over $k$, via the natural action of the Galois group
$\Gal(\kbar/k)$ on $\Pic(S_\kbar)$. Since this action respects the
intersection pairing, it factorizes via a subgroup $W(S)$ of
$W(R_d)$. The~group $W(S)$ acts on the set of $(-1)$-curves, breaking it
into $n$ orbits\break \mbox{$\{E_{i,1}, \dots, E_{i,k_i}\}$}~of size $k_i$, for
$i=1, \dots, n$.

Since $\alpha(S)$ is irrelevant if $S(k) = \emptyset$, we assume once and for
all that $S$ has a \mbox{$k$-rational}~point.  Then $\smash{\Pic(S) =
  \Pic(S_{\overline{k}})^{\Gal(\overline{k}/k)}}$ by \cite[Theorem~2.1.2,
Claim~(iii)]{CTS}. Consequently, $\Eff(S)$ is generated by the classes of
$E_{i,1}+\dots+E_{i,k_i}$, for $i=1, \dots, n$.

Therefore, $\alpha(S)$ depends only on the conjugacy class of $W(S)$ in
$W(R_d)$, which reduced the computation of all cases of $\alpha(S)$ to a
finite problem. We denote the number of conjugacy classes of subgroups of
$W(R_d)$ by $c_d$, see Table~\ref{tab:weyl_groups}. More precisely, it depends
only on the orbit structure of the action of $W(S)$ on the set of
$(-1)$-curves, up to conjugacy, reducing the problem to $c_d'$ classes in
degree $d$.

For $d \le 6$, the sum of the classes of the $(-1)$-curves in $\Pic(S_\kbar)$
is $\frac{N_d}{d}\cdot (-K_S)$. Indeed, this sum is invariant under $W(R_d)$,
so it is a scalar multiple of $(-K_S)$. The value of this scalar is determined
from the fact that $(-K_S,-K_S)=d$ and $([E],-K_S)=1$ for each $(-1)$-curve
$E$.

\begin{theorem}
\label{inkl}
Let\/~$S_1, S_2$ be del Pezzo surfaces of degree\/ $d \le \!7$ over a
field\/~$k$ such that\/ $S_1(k), S_2(k) \neq \emptyset$.  For the
corresponding subgroups\/ $W(S_1), W(S_2)$ of\/ $W(R_d)$, suppose that\/
$W(S_1) \subseteq gW(S_2)g^{-1}$ for some\/ $g \in W(R_d)$.  Further,~assume\/
$\rk \Pic(S_1) = \rk \Pic(S_2)$.  Then\/~$\alpha(S_1) = \alpha(S_2)$.
\end{theorem}

\begin{proof}
  As~$\Pic(S_i) = \Pic((S_i)_{\kbar})^{W(S_i)}$, we see that $\Pic(S_i)$ are
  maximal~sublattices. On~the other hand, up to isometry, $\Pic(S_1) \supseteq
  \Pic(S_2)$.  Thus,~the equality of the ranks implies that $\Pic(S_1) =
  \Pic(S_2)$.  Further,~every \mbox{$W(S_2)$-orbit} of \mbox{$(-1)$-curves}
  breaks into one or several \mbox{$W(S_1)$-orbits}.  If~$\{E_{1,1}, \ldots,
  E_{k,l_k}\}$ is a \mbox{$W(S_2)$-orbit} that~breaks into the
  \mbox{$W(S_1)$-orbits} $\{E_{i,1}, \ldots, E_{i,l_i}\}$, for $i =
  1,\ldots,k$,~then\break $[E_{i,1}] + \ldots + [E_{i,l_i}] \in \Pic(S_1) =
  \Pic(S_2)$ is $W(S_2)$-invariant and therefore independent of~$i$.
  In~particular,
  \begin{equation*}
    [E_{i,1}] + \ldots + [E_{i,l_i}] = \frac1k ([E_{1,1}] +
    \ldots + [E_{k,l}]) \in \Pic(S_1)_\bbR.
  \end{equation*}
  Hence,~$\Eff(S_1) = \Eff(S_2)$.
\end{proof}

\begin{ttt}\label{maximal}
  Let $S$ be of degree $d$.  The stabilizer $G$ of $\Pic(S)$ for the
  action of $W(R_d)$ on $\Pic(S_\kbar)$ is the unique subgroup of
  $W(R_d)$ containing $W(S)$ that is maximal with the property that
  $\varrho(S) = \rk(\Pic(S_\kbar)^G)$. We call each such $G$ a
  \emph{$\varrho$-maximal subgroup} of $W(R_d)$. Let $c_d''$ be the
  number of conjugacy classes of $\varrho$-maximal subgroups of
  $W(R_d)$ (cf.\ Table~\ref{tab:weyl_groups}).  Theorem~\ref{inkl}
  shows that it is enough to compute $\alpha(S)$ in the corresponding
  $c_d''$ cases.
\end{ttt}

\section{The algorithm}

\begin{algo}[$\alpha(S)$ for del Pezzo surfaces $S$ of degree $d \leq 5$]
\label{poly}
\begin{III}
\item Realize in {\tt gap} or {\tt magma} the group $W(R_d)$ as a
  subgroup of the symmetric group~$\SSS_{N_d}$, identifying the set of
  all lines with $\{1, \dots, N_d\}$.
\item Determine the $N_d \times N_d$ intersection matrix~$M_d$, using
  Remark~\ref{rem:matrix}.
\item The~geometric Picard~group of rank $10-d$ is now isomorphic to
  $\bbZ^{N_d} / \ker M_d$.  Select $(-1)$-curves $E_1, \ldots,
  E_{9-d}$ that are pairwise skew and a $(-1)$-curve~$E_{10-d}$ that
  is distinct from the previous ones. Ensure~that the corresponding
  $(10-d) \times (10-d)$ intersection matrix is of determinant~$\pm 1$.
  Then $\{[E_1], \ldots, [E_{10-d}]\}$ forms a basis of the geometric
  Picard~group. Express~every line in this~basis. This~yields a map $p
  \colon \{1, \dots, N_d\} \to \bbZ^{10-d}$, represented by a~matrix.
\item\label{vier} Calculate a set $U_d$ of representatives of the
  $c_d$ conjugacy classes of subgroups in~$W(R_d)$, together with the
  number of subgroups in each conjugacy class.
\item\label{it:ranks} For each subgroup $G \in U_d$, compute the $n_G$
  orbits of $G$ acting on $\{1, \dots, N_d\}$. Compute the rank
  $\varrho_G$ of $\Pic(S_\kbar)^G$, which is the rank of the $n_G
  \times n_G$ matrix whose entry $(i,j)$ is the intersection number
  (computed using $M_d$) of the sums of elements in the $i$-th and
  $j$-th orbit.
\item\label{it:reduce} Repeat the following until $U_d$ is empty, producing a set $M_d$
  (initially empty, at the end of order $c_d''$) of subgroups of
  $W(R_d)$ such that for any subgroup $H$ there is a unique $G \in
  M_d$ with $\varrho_G = \varrho_H$ containing a conjugate of $H$.
  \begin{iii}
  \item Choose a subgroup $G \in U_d$ of maximal order and add it to
    $M_d$.
  \item Remove all $H$ from $U_d$ with $\varrho_H = \varrho_G$
    contained in a conjugate of $G$.
  \end{iii}
\item\label{it:alpha} For each $G \in M_d$, compute the corresponding value of
  $\alpha(S)$ as follows:
  \begin{iii}
  \item For~each orbit $\{l_{i,1}, \ldots, l_{i,k_i}\}$ of $G$ acting on $\{1,
    \dots, N_d\}$ (with $i=1, \dots, n_G$), calculate the vector~$v_i =
    p(l_{i,1}) + \cdots + p(l_{i,k_i}) \in \bbZ^{10-d}$.  This~yields a list
    $v_1, \ldots, v_{n_G}$ of vectors in~$\bbZ^{10-d}$.
  \item Determine a basis of the free \mbox{$\bbZ$-module}
    $\bbZ^{10-d} \cap (\bbQ v_1 + \cdots + \bbQ v_{n_G}) \cong
    \bbZ^\varrho$, which is isomorphic to the
    Picard~group. Express~the vectors $v_1, \ldots, v_{n_G}$ in this basis
    and print the list of coefficient vectors $w_1, \ldots, w_{n_G} \in
    \bbZ^\varrho$ obtained, together with a marker of the conjugacy
    class treated, into a~file.
  \item Read~this file into a {\tt polymake}~script. For~each
    conjugacy class, realize in {\tt polymake} the polytope
    in~$\bbR^\varrho$, given by $\langle x, w_1 \rangle \geq 0, \ldots
    , \langle x, w_{n_G} \rangle \geq 0$ and\break $\langle x, w_1 + \cdots +
    w_{n_G} \rangle \leq \frac{N_d}d$.  Compute~the volume of the
    polytope, multiply by~$t$, and return the~result.
  \end{iii}
\end{III}
\end{algo}

\begin{remarks}
  \begin{abc}
  \item Steps~\ref{it:ranks} and~\ref{it:reduce} use Theorem~\ref{inkl} to
    reduce the number of computations of $\alpha(S)$ significantly from $c_d$
    to $c_d''$. This leads to a reasonably short list of results with a
    natural structure. For $d=1$, this also seems absolutely necessary to keep
    the running times reasonably low.

    To find all $H$ with $\varrho_H = \varrho_G$ contained in a conjugate of
    $G$ in $W=W(R_d)$ in \texttt{magma}, one may either compute
    \texttt{Conjugates(W,G)} and compare all candidates $H$ with the resulting
    list for inclusion, or test \texttt{IsConjugate(W,H,U)} for all $U$ in
    \texttt{Subgroups(G)}. For $d=1$, depending on the number of conjugates of
    subgroups of $G$ (which is up to $604800$) and candidates $H$ (which is up
    to $48797$), each option might take prohibitively long. It turns out that
    it is reasonable to take the first approach whenever the number of
    conjugates, known from the computation of \texttt{Subgroups(W)} in
    step~\ref{vier}, is less than the number of candidates. With this
    strategy, we must only deal with up to $1120$ conjugates in the first case
    and up to $1886$ candidates in the second case.
  \item Computationally, the case of degree $d=1$ is by far the
    hardest. Using \texttt{magma} V2.17-9 on an Intel Xeon L5640 CPU
    at 2.27 GHz, step~\ref{vier} took $32$ minutes,
    step~\ref{it:ranks} took $68$ minutes and step~\ref{it:reduce}
    took 3 minutes, resulting in $41$ conjugacy classes of
    $\varrho$-maximal subgroups. For $d=2$,
    steps~\ref{vier}--\ref{it:reduce} took a total of $35$ seconds
    giving $32$ conjugacy classes, and $d\ge 3$ is negligible.

    For $d \ge 3$, one can also use \texttt{gap} to compute the
    conjugacy classes of subgroups together with a numbering from $1$
    to~$c_d$ in~step~\ref{vier}. This~numbering is reproducible, at
    least in our version of~{\tt gap}. For $W(E_7)$ and~$W(E_8)$,
    however, \texttt{gap} runs out of memory. The~difference comes
    from the fact that the Cannon/Holt algorithm~\cite{CH} to
    determine the maximal subgroups of a given finite group is
    available in~{\tt magma}.
  \item The~running times for the volume computations were as~follows, using
    {\tt polymake}, version 2.9.9, on~an AMD Phenom II X4 955
    processor. We~describe the polytope by its {\tt INEQUALITIES}~properties.
    Again, degrees $d \ge 3$ are done in a few seconds.
   
    The~$32$ cases of del Pezzo surfaces of degree $2$ together took
    $85$~seconds of CPU~time.  Among~them, the most complicated case is
    that of the split del Pezzo~surface.  Here,~according to the
    definition, one has to compute the volume of a polytope
    in~$\bbR^8$, having $703$~vertices. This~case alone took 36~seconds;
    279\,MBytes of memory were being~allocated.

    The~39 cases of del Pezzo surfaces of degree $1$, excluding those
    of Picard ranks $8$ or~$9$, together took $140$~seconds of
    CPU~time. Here,~$300$\,MBytes of memory were~allocated.
    The~rank-$8$ case alone, however, took around $37$~minutes of CPU
    time and $3$\,GBytes of~memory.  Finally,~the split del Pezzo
    surface of degree one leads to a polytope in~$\bbR^9$ with
    $19\,441$~vertices. Here,~{\tt polymake} fails, as $8$\,GBytes of
    working memory turn out to be~insufficient. In~order to make it
    work, we incorporated the obvious $\SSS_8$-symmetry of
    the~polytope. Then~the volume could be computed within
    $3.8$~seconds, using only $150$\,MBytes of~memory. The same trick
    works for the rank-$8$ case. Here,~incorporating the obvious
    $\SSS_6$-symmetry reduces the running time to $4.8$~seconds and the
    memory usage to $158$\,MBytes.
  \item In~the cases when $S$ is isomorphic over~$k$ to the blow-up
    of~$\bP^2$ in some Galois-invariant set of size $9-d$, one has
    explicit generators for the Picard
    group~\cite[Theorem~V.4.9]{Ha}. Such~cases may, in principle, be
    handled~interactively. For~example, the following snippet of {\tt
    polymake}~code computes $\alpha(S)$ for a split cubic
    surface~$S$.

{\tiny
\begin{verbatim}
$p=new Polytope<Rational>(INEQUALITIES=>[[0,0,1,0,0,0,0,0], [0,0,0,1,0,0,0,0], [0,0,0,0,1,0,0,0],
[0,0,0,0,0,1,0,0], [0,0,0,0,0,0,1,0], [0,0,0,0,0,0,0,1], [0,1,-1,-1,0,0,0,0], [0,1,-1,0,-1,0,0,0],
[0,1,-1,0,0,-1,0,0], [0,1,-1,0,0,0,-1,0], [0,1,-1,0,0,0,0,-1], [0,1,0,-1,-1,0,0,0],
[0,1,0,-1,0,-1,0,0], [0,1,0,-1,0,0,-1,0], [0,1,0,-1,0,0,0,-1], [0,1,0,0,-1,-1,0,0],
[0,1,0,0,-1,0,-1,0], [0,1,0,0,-1,0,0,-1], [0,1,0,0,0,-1,-1,0], [0,1,0,0,0,-1,0,-1],
[0,1,0,0,0,0,-1,-1], [0,2,-1,-1,-1,-1,-1,0], [0,2,-1,-1,-1,-1,0,-1], [0,2,-1,-1,-1,0,-1,-1], 
[0,2,-1,-1,0,-1,-1,-1], [0,2,-1,0,-1,-1,-1,-1], [0,2,0,-1,-1,-1,-1,-1], [1,-3,1,1,1,1,1,1]]);
print (($p->DIM)*($p->VOLUME));
\end{verbatim}
}
  \end{abc}
\end{remarks}

\section{Results}

For~$d \geq 5$, the values of $\alpha(S)$ were systematically computed
in~\cite[Section~7B]{DJT}. For~$d = 5$, our algorithm recovers the results
listed in \cite[Table~8]{DJT}. For $d=4$, the description of the results is
relatively straightforward. For the probably most important and interesting
case $d=3$ of cubic surfaces, we give more details. For simplicity, we only
give an overview of the results for $d=2$ and $d=1$.

\begin{theorem}[The values of $\alpha(S)$ for quartic del Pezzo surfaces]
  Let\/~$S$ be a smooth quartic del Pezzo surface over a
  field\/~$k$ such that\/~$S(k) \neq \emptyset$.  Then exactly one of
  following is~true.\smallskip

\begin{III}
\item $\rk \Pic(S) = 1$.  Then $\alpha(S) = 1$.\smallskip
\item $\rk \Pic(S) = 2$.  Then there are two~cases.
\begin{iii}
\item $S$~has no\/ \mbox{$k$-rational} $(-1)$-curve.  Then $\alpha(S)
  = 1$.
\item $S$~is isomorphic to\/ $\bP^2$, blown up in an orbit of size~five.
  Then $\alpha(S) = \frac23$.
\end{iii}\smallskip
\item $\rk \Pic(S) = 3$.  Then there are two~cases.
\begin{iii}
\item $S$~has no\/ \mbox{$k$-rational} $(-1)$-curve.  Then $\alpha(S) =
  \frac12$.
\item $S$~is isomorphic to\/ $\bP^2$, blown up in a\/ \mbox{$k$-rational}
  point and an orbit of size four or in an orbit of size two and an orbit of
  size~three.  Then $\alpha(S) = \frac13$.
\end{iii}\smallskip
\item $\rk \Pic(S) = 4$.  Then there are two~cases.
\begin{iii}
\item $S$~has no\/ \mbox{$k$-rational} $(-1)$-curve.  Then $\alpha(S) =
  \frac16$.
\item $S$~is isomorphic to\/ $\bP^2$, blown up in two\/ \mbox{$k$-rational}
  points and an orbit of size three or in a\/ \mbox{$k$-rational} point and
  two orbits of size~two.  Then $\alpha(S) = \frac19$.
\end{iii}\smallskip
\item $\rk \Pic(S) = 5$.  Then $\alpha(S) = \frac1{36}$.\smallskip
\item $\rk \Pic(S) = 6$.  Then $\alpha(S) = \frac1{180}$.
\end{III}
\end{theorem}

\begin{proof}
  Our~implementation of Algorithm~\ref{poly}
    (skipping the reduction step~\ref{it:reduce} and working $U_d$ instead of
    $M_d$ afterwards) yields a list, associating to each number from 1 to~197
    a value of~alpha. The~result is obtained by giving a geometric
    interpretation to this~list.
\end{proof}

\noindent
  Step~\ref{it:reduce} of Algorithm~\ref{poly} gives $14$ conjugacy
  classes of $\varrho$-maximal subgroups (see \ref{maximal}) of
  $W(D_5)$. In step~\ref{it:alpha}, we discover that this leads to
  eight distinct values of $\alpha(S)$. Among~the $197$ conjugacy
  classes of subgroups and the $38$ orbit structures, they are
  distributed as shown in Table~\ref{tab:degfour}.

\begin{table}[ht]
\begin{tabular}{|l|c||c|c||c|c||c|}
  \hline
  Case & $\alpha$ & \!\!\#conj.\!\! & $\varrho$-maximal & \#orbit        & maximal                       & corr.\,case \\
  &          &   classes       &         & \!structures\! & \phantom{[1,1,1,1,2,2,2,2,4]} & cub.\,surf. \\\hline
  \hline\rule{0pt}{2.2ex}       
  I       & $1$             &            98 & $W(D_5)$                                 & \phantom{0}7 & [16]                    & II.i       \\[0.5mm]\hline\hline\rule{0pt}{2.2ex}
  II.i    & $1$             &            50 & $\SSS_4 \rtimes\, (\bbZ/2\bbZ)^3$         &           12 & [8,8]                   & III.i \\[0.5mm]\hline\rule{0pt}{2.2ex}
  &                 &  \phantom{0}7 & $\SSS_3 \times\, \SSS_2 \times \bbZ/2\bbZ$ & \phantom{0}1 & [2,2,6,6]               & III.iii \\[0.5mm]\hline\rule{0pt}{2.2ex}
  &                 &            11 & $\SSS_4 \times \SSS_2$                     & \phantom{0}2 & [4,4,8]                 & III.iv \\[0.5mm]\hline\rule{0pt}{2.2ex}
  II.ii   & $\frac23$       &  \phantom{0}5 & $\SSS_5$                                  & \phantom{0}2 & [1,5,10]                & III.v \\[0.5mm]\hline\hline\rule{0pt}{2.2ex}
III.i   & $\frac12$       &  \phantom{0}5 & $\SSS_4$                                  & \phantom{0}1 & [4,4,4,4]               & IV\!.i \\[0.5mm]\hline\rule{0pt}{2.2ex}
        &                 &  \phantom{0}5 & $\SSS_2 \times \SSS_2 \times \SSS_2$        & \phantom{0}3 & [2,2,2,2,4,4]           & IV\!.ii \\[0.5mm]\hline\rule{0pt}{2.2ex}
III.ii  & $\frac13$       &  \phantom{0}5 & $\SSS_4$                                  & \phantom{0}3 & [1,1,4,4,6]             & IV\!.i \\[0.5mm]\hline\rule{0pt}{2.2ex}
        &                 &  \phantom{0}3 & $\SSS_2 \times \SSS_3$                     & \phantom{0}1 & [1,1,2,3,3,6]           & IV\!.iii \\[0.5mm]\hline\hline\rule{0pt}{2.2ex}
IV\!.i  & $\frac16$       &  \phantom{0}2 & $\SSS_2 \times \SSS_2$                     & \phantom{0}1 & [2,2,\ldots,2]          & V\!.i \\[0.5mm]\hline\rule{0pt}{2.2ex}
IV\!.ii & $\frac19$       &  \phantom{0}2 & $\SSS_2 \times \SSS_2$                     & \phantom{0}2 & [1,1,1,1,2,2,2,2,4]     & V\!.i \\[0.5mm]\hline\rule{0pt}{2.2ex}
        &                 &  \phantom{0}2 & $\SSS_3$                                  & \phantom{0}1 & [1,1,1,1,3,3,3,3]       & V\!.ii \\[0.5mm]\hline\hline\rule{0pt}{2.2ex}
V       & $\frac1{36}$    &  \phantom{0}1 & $\SSS_2$                                  & \phantom{0}1 & [1,\ldots,1,2,2,2,2]    & VI \\[0.5mm]\hline\hline\rule{0pt}{2.2ex}
VI      & $\frac1{180}$   &  \phantom{0}1 & $0$                                      & \phantom{0}1 & [1,1,\ldots,1]          & VII \\[0.5mm]\hline
\hline
~~$\sum$&                 & 197\phantom{0} & \phantom{$\SSS_3 \times \SSS_3 \times \bbZ/2\bbZ$} & \phantom{000}38\phantom{000} & & \phantom{corr.\,case} \\\hline
\end{tabular}\vskip2mm
\caption{Degree four del Pezzo surfaces, the $14$ $\varrho$-maximal cases}\label{tab:degfour}\vskip-5mm
\end{table}

The rightmost column in this table indicates the type of cubic surface
(cf. Theorem~\ref{thm:cubic} and Table~\ref{Tab} below) that occurs
when blowing up one \mbox{$k$-rational}~point.  It~is a little
surprising that the types IV\!.i and~V\!.i appear~twice. The~reason
for this is as~follows.

A quartic del Pezzo surface of type III.i may be constructed by blowing up
$\bP^2$ in two \mbox{$k$-rational}~points and an orbit of size four, followed
by blowing down the line through the two \mbox{$k$-rational}~points.  This~is
a non-blown-up case, but \mbox{$k$-birationally} equivalent to a surface of
type~III.ii, which may be obtained by blowing up $\bP^2$ in a
\mbox{$k$-rational}~point and an orbit of size~four.
For~the cases IV\!.i and~IV\!.ii, the situation is very~similar.

\begin{theorem}[The values of $\alpha(S)$ for cubic surfaces]\label{thm:cubic}
  Let\/~$S$ be a smooth cubic surface over a field\/~$k$ such
  that\/~$S(k) \neq \emptyset$.  Then exactly one of following
  is~true.\smallskip

\begin{III}
\item $\rk \Pic(S) = 1$.  Then $\alpha(S) = 1$.\smallskip
\item $\rk \Pic(S) = 2$.  Then there are four~cases.
\begin{iii}
\item $S$~has a\/ \mbox{$k$-rational} line.  Then $\alpha(S) = 1$.
\item $S$~is isomorphic to\/ $\bP^2$, blown up in an orbit of
  size~six.  Then $\alpha(S) = \frac43$.
\item $S$~has a Galois-invariant double-six~\cite{EJ2}. Over~the
  quadratic extension\/~$l/k$, splitting the double-six, $S_l$ is
  isomorphic to\/ $\bP^2$, blown up in two orbits of size~three.
  Then $\alpha(S) = 2$.
\item $S$~has a Galois orbit consisting of two skew~lines.
  Then $\alpha(S) = \frac32$.
\end{iii}\smallskip
\item $\rk \Pic(S) = 3$.  Then there are five~cases.
\begin{iii}
\item $S$~has three coplanar\/ \mbox{$k$-rational}~lines.
  Then $\alpha(S) = \frac12$.
\item $S$~is isomorphic to\/ $\bP^2$, blown up in two orbits of
  size~three.  Then $\alpha(S) = 1$.
\item $S$~has a Galois-invariant double-six. Over~the quadratic
  extension\/~$l/k$, splitting the double-six, $S_l$ is isomorphic
  to\/ $\bP^2$, blown up in a\/ \mbox{$k$-rational} point, an orbit of
  size two, and an orbit of size~three.  Then, $\alpha(S) = 1$.
\item $S$~is isomorphic to\/ $\bP^2$, blown up in an orbit of size two
  and an orbit of size~four.  Then $\alpha(S) = \frac56$.
\item $S$~is isomorphic to\/ $\bP^2$, blown up in a\/
  \mbox{$k$-rational} point and an orbit of size~five.  Then,
  $\alpha(S) = \frac{17}{24}$.
\end{iii}\smallskip
\item $\rk \Pic(S) = 4$.  Then there are three~cases.
\begin{iii}
\item $S$~is isomorphic to\/ $\bP^2$, blown up in two\/
  \mbox{$k$-rational} points and an orbit of size~four.
  Then $\alpha(S) = \frac5{18}$.
\item $S$~is isomorphic to\/ $\bP^2$, blown up in three orbits of
  size~two.  Then $\alpha(S) = \frac7{18}$.
\item $S$~is isomorphic to\/ $\bP^2$, blown up in a\/
  \mbox{$k$-rational} point, an orbit of size two, and an orbit of
  size~three.  Then $\alpha(S) = \frac38$.
\end{iii}\smallskip
\item $\rk \Pic(S) = 5$.  Then there are two~cases.
\begin{iii}
\item $S$~is isomorphic to\/ $\bP^2$, blown up in two\/
  \mbox{$k$-rational} points and two orbits of size~two.
  Then $\alpha(S) = \frac18$.
\item $S$~is isomorphic to\/ $\bP^2$, blown up in three\/
  \mbox{$k$-rational} points and an orbit of size~three.
  Then $\alpha(S) = \frac5{48}$.
\end{iii}\smallskip
\item $\rk \Pic(S) = 6$.  Then $\alpha(S) = \frac1{30}$.\smallskip
\item $\rk \Pic(S) = 7$.  Then $\alpha(S) = \frac1{120}$.
\end{III}
\end{theorem}

\noindent
According~to Theorem~\ref{thm:cubic}, there are $17$ conjugacy classes of
$\varrho$-maximal subgroups, leading to 14 distinct values of~$\alpha(S)$. The
$350$ conjugacy classes of subgroups and the 91 orbit structures are
distributed among them as shown in Table~\ref{Tab}. In~case III.i,
$(\bbZ/2\bbZ)^3$ means the sum zero subspace in~$(\bbZ/2\bbZ)^4$, acted upon
by~$\SSS_4$ in the obvious~manner.

\begin{table}[ht]
\begin{tabular}{|l|c||c|c||c|c|}
  \hline
  Case & $\alpha$ & \!\!\#conjugacy\!\! & $\varrho$-maximal & \#orbit            & maximal \\
  &          &   classes           &         & \!\!structures\!\! &         \\\hline
  \hline\rule{0pt}{2.2ex}       
  I       & $1$             &           137 & $W(E_6)$                                 &           22 & [27]                           \\[0.5mm]\hline\hline\rule{0pt}{2.2ex}
  II.i    & $1$             & \phantom{0}98 & $W(D_5)$                                 &           22 & [1,10,16]                      \\[0.5mm]\hline\rule{0pt}{2.2ex}
  II.ii   & $\frac43$       & \phantom{0}16 & $\SSS_6$                                  & \phantom{0}5 & [6,6,15]                       \\[0.5mm]\hline\rule{0pt}{2.2ex}
II.iii  & $2$             & \phantom{0}11 & $\SSS_3 \times \SSS_3 \times\, \bbZ/2\bbZ$ & \phantom{0}3 & [3,3,6,6,9]                    \\[0.5mm]\hline\rule{0pt}{2.2ex}
II.iv   & $\frac32$       & \phantom{00}8 & $\SSS_5 \times\, \bbZ/2\bbZ$              & \phantom{0}2 & [2,5,10,10]                    \\[0.5mm]\hline\hline\rule{0pt}{2.2ex}
III.i   & $\frac12$       & \phantom{0}33 & $\SSS_4 \rtimes\, (\bbZ/2\bbZ)^3$         &           13 & [1,1,1,8,8,8]                  \\[0.5mm]\hline\rule{0pt}{2.2ex}
III.ii  & $1$             & \phantom{00}6 & $\SSS_3 \times \SSS_3$                     & \phantom{0}3 & [3,3,3,3,3,3,9]                \\[0.5mm]\hline\rule{0pt}{2.2ex}
III.iii & $1$             & \phantom{00}7 & $\SSS_3 \times \SSS_2 \times\, \bbZ/2\bbZ$ & \phantom{0}2 & \phantom{000}[1,2,2,3,3,4,6,6]\phantom{000} \\[0.5mm]\hline\rule{0pt}{2.2ex}
III.iv  & $\frac56$       & \phantom{0}11 & $\SSS_4 \times \SSS_2$                     & \phantom{0}5 & [1,2,2,4,4,6,8]                \\[0.5mm]\hline\rule{0pt}{2.2ex}
III.v   & $\frac{17}{24}$ & \phantom{00}5 & $\SSS_5$                                  & \phantom{0}2 & [1,1,5,5,5,10]                 \\[0.5mm]\hline\hline\rule{0pt}{2.2ex}
IV\!.i  & $\frac5{18}$    & \phantom{00}5 & $\SSS_4$                                  & \phantom{0}3 & [1,1,1,1,1,4,4,4,4,6]          \\[0.5mm]\hline\rule{0pt}{2.2ex}
IV\!.ii & $\frac7{18}$    & \phantom{00}4 & $\SSS_2 \times \SSS_2 \times \SSS_2$        & \phantom{0}3 & [1,1,1,2,2,2,2,2,2,4,4,4]      \\[0.5mm]\hline\rule{0pt}{2.2ex}
IV\!.iii& $\frac38$       & \phantom{00}3 & $\SSS_3 \times \SSS_2$                     & \phantom{0}1 & [1,1,1,2,2,2,3,3,3,3,6]        \\[0.5mm]\hline\hline\rule{0pt}{2.2ex}
V\!.i   & $\frac18$       & \phantom{00}2 & $\SSS_2 \times \SSS_2$                     & \phantom{0}2 & [1,1,1,1,1,1,1,2,2,\ldots,2,4] \\[0.5mm]\hline\rule{0pt}{2.2ex}
V\!.ii  & $\frac5{48}$    & \phantom{00}2 & $\SSS_3$                                  & \phantom{0}1 & [1,1,\ldots,1,3,3,3,3,3,3]     \\[0.5mm]\hline\hline\rule{0pt}{2.2ex}
VI      & $\frac1{30}$    & \phantom{00}1 & $\SSS_2$                                  & \phantom{0}1 & [1,1,\ldots,1,2,2,2,2,2,2]     \\[0.5mm]\hline\hline\rule{0pt}{2.2ex}
VII     & $\frac1{120}$   & \phantom{00}1 & $0$                                      & \phantom{0}1 & [1,1,\ldots,1]                 \\[0.5mm]\hline
\hline
~~$\sum$&                 & \phantom{000}350\phantom{000} & \phantom{$\SSS_3 \times \SSS_3 \times \bbZ/2\bbZ$} & \phantom{000}91\phantom{000} & \\\hline
\end{tabular}\vskip2mm
\caption{Smooth cubic surfaces, the $17$ $\varrho$-maximal cases}\label{Tab}\vskip-5mm
\end{table}

\begin{remarks}
\begin{abc}
\item Some~of the cases allow equivalent~characterizations.
  For~example, IV\!.i contains the cubic surfaces of Picard rank~$4$
  having five rational lines that form two triangles with a line
  in~common.
\item Observe~that $\rk\Pic(S) \geq 4$ implies that $S$ has a
  Galois-invariant~sixer, i.e.,~$S$ is then isomorphic over~$k$ to the
  blow-up of~$\bP^2$ in some Galois-invariant set of size~six.
\item Recall~that {\tt gap} produces a list, giving the $350$ conjugacy
  classes of subgroups of~$W(E_6)$ in a definite~numbering.  Associated to
  each number, we have a value $\alpha(S)$.  Unfortunately, this list is far
  too long to be reproduced here. But let us restrict considerations to
  $\rk\Pic(S)\geq 2$ and one group per orbit structure, the maximal
  one. Then~the situation may be visualized by the inclusion graphs in
  Figure~\ref{fig:inclusion}.

\begin{figure}[ht]
\centerline{
\includegraphics[width=10.8cm]{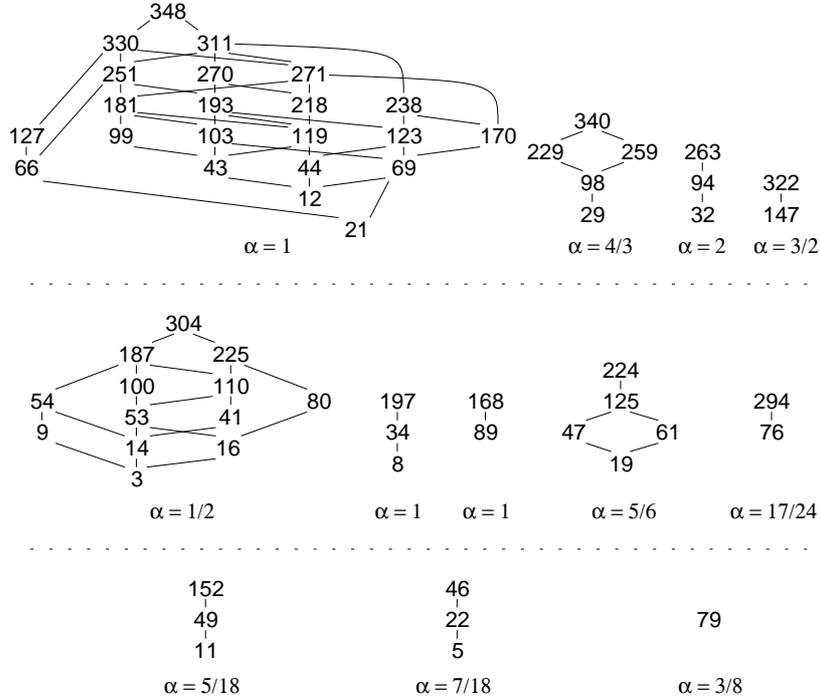}
}
\caption{Conjugacy classes of subgroups with Picard ranks $2$, $3$, and~$4$, only one per orbit structure, numbered as in {\tt gap} 4.4.12.}\label{fig:inclusion}
\end{figure}

The~four conjugacy classes with Picard rank~$5$ are $4$ and~$10$, leading
to~$\alpha = \frac18$, as~well as $7$ and~$24$, leading to~$\alpha =
\frac5{48}$.  Here,~$7$ and~$24$ have the same
orbit~structure. Finally,~number~$2$ is the only with Picard rank~$6$ and
number~$1$ the only with Picard rank~$7$.
\end{abc}
\end{remarks}

\begin{example}
\label{S5}
Let us take a closer look at case~III.v.  Here,~$S$ is obtained by
blowing up $\bP^2$ in a rational point and an orbit of
size~five. The~Galois group must permute the five
points~transitively. Therefore,~the maximal possible Galois group is
isomorphic to~$\SSS_5$.  It~has {\tt gap} number~$294$.

Further,~$\SSS_5$ has exactly five conjugacy classes of transitive
subgroups. This~explains the number $5$ in the third column
of~Table~\ref{Tab}. Among~the transitive subgroups, there are the
cyclic group of order~five and the dihedral group of order~$10$.
For~these, from the explicit description of the $27$~lines
\cite[Theorem~V.4.9]{Ha}, one easily deduces that the finer orbit
structure [1,1,5,5,5,5,5]~occurs. In~fact, the dihedral group is the
maximal subgroup corresponding to this orbit~structure. Its~{\tt gap}
number is~$76$.
\end{example}

\begin{ttt}
  For~$d=2$ and $d=1$,~Theorem~\ref{inkl} reduces the problem to $41$
  resp. $32$ conjugacy classes of $\varrho$-maximal subgroups. It~is
  therefore possible to specify all the values of $\alpha(S)$ in
  Table~\ref{tab:degtwoone}. The~subscripts $2$ shall indicate that
  the corresponding rational number arises in two distinct~cases.

\begin{table}[ht]
\centerline{
\begin{tabular}{|c|l|}
\hline
Picard & \\
rank   & $\alpha$ \\\hline\hline\rule{0pt}{2.2ex}
$1$ & $1$ \\[0.5mm]\hline\rule{0pt}{2.2ex}
$2$ & $1, 2_2, \frac7{3}, 3_2, 4$ \\[0.5mm]\hline\rule{0pt}{2.2ex}
$3$ & $1, \frac5{3}, \frac{11}6, 2, \frac9{4}, \frac5{2}, \frac8{3}, 3_2$ \\[0.5mm]\hline\rule{0pt}{2.2ex}
$4$ & $\frac2{3}, \frac{11}{12}, \frac{11}9, \frac{13}9, \frac{19}{12}, \frac5{3}, 2$ \\[0.5mm]\hline\rule{0pt}{2.2ex}
$5$ & $\frac{17}{36}, \frac2{3}, \frac{13}{18}, 1$ \\[0.5mm]\hline\rule{0pt}{2.2ex}
$6$ & $\frac7{30}, \frac{13}{45}$ \\[0.5mm]\hline\rule{0pt}{2.2ex}
$7$ & $\frac1{10}$ \\[0.5mm]\hline\rule{0pt}{2.2ex}
$8$ & $\frac1{30}$ \\[0.5mm]\hline
\end{tabular}
\qquad
\raisebox{-0.31cm}{
\begin{tabular}{|c|l|}
\hline
Picard & \\
rank   & $\alpha$ \\\hline\hline\rule{0pt}{2.2ex}
$1$ & $1$ \\[0.5mm]\hline\rule{0pt}{2.2ex}
$2$ & $2, 4_2, \frac{16}3, 6, 7, 8, 10$ \\[0.5mm]\hline\rule{0pt}{2.2ex}
$3$ & $3, 4, 6, \frac{77}{12}, \frac{23}3, \frac{26}3, 9, \frac{32}3, 11_2, 14$ \\[0.5mm]\hline\rule{0pt}{2.2ex}
$4$ & $4, \frac{35}6, \frac{20}3, \frac{31}4, \frac{85}9, \frac{92}9, \frac{31}3, 13$ \\[0.5mm]\hline\rule{0pt}{2.2ex}
$5$ & $4, \frac{355}{72}, \frac{41}6, \frac{31}4, \frac{103}{12}, \frac{92}9$ \\[0.5mm]\hline\rule{0pt}{2.2ex}
$6$ & $\frac{178}{45}, \frac{16}3, \frac{94}{15}$ \\[0.5mm]\hline\rule{0pt}{2.2ex}
$7$ & $\frac{59}{20}, \frac{18}5$ \\[0.5mm]\hline\rule{0pt}{2.2ex}
$8$ & $\frac{29}{15}$ \\[0.5mm]\hline\rule{0pt}{2.2ex}
$9$ & $1$ \\[0.5mm]\hline
\end{tabular}
}
}\vskip2mm
\caption{Del Pezzo surfaces of degrees $2$ and~$1$, the $32$ resp.\ $41$ $\varrho$-maximal cases}\label{tab:degtwoone}
\end{table}\vskip-5mm
\end{ttt}

\begin{remarks}\label{rem:orbits}
\begin{abc}
\item When verifying Manin's conjecture
  \begin{equation*}
    N_{U,H}(B) \sim c_{S,H} B(\log B)^{\varrho(S)-1}
  \end{equation*}
  for a del Pezzo surface $S$ of degree $d$ over a number field $k$,
  the rank $\varrho(S)$ of $\Pic(S)$ must be determined, and this is
  usually done via the action of $\Gal(\kbar/k)$ on the set of
  $(-1)$-curves. With this information, the factor $\alpha(S)$ of
  $c_{S,H}$ can be read off our tables as follows.

  If $S$ is smooth of degree $d \ge 4$ over $k$, we note that the value of
  $\alpha(S)$ is uniquely determined by $d$, $\varrho(S)$ and the question
  whether $S$ contains at least one line defined over~$k$.

  For $d=3$, the situation is slightly more complicated. With one
  exception, $\alpha(S)$ can be read off once one has determined
  $\varrho(S)$, the number of lines defined over $k$ and the numerical
  orbit structure (i.e., the number of elements in each
  $\Gal(\kbar/k)$-orbit on the set of $(-1)$-curves). Indeed, an
  analysis of our data shows that $\alpha(S)$ has the same value as
  the unique one of the $17$ $\varrho$-maximal cases in
  Table~\ref{Tab} with the same $\varrho(S)$, the same number of
  $(-1)$-curves defined over $k$, and an orbit structure that can be
  split into the given orbit structure (cf.\ Theorem~\ref{inkl}), with
  the following exception.

  Numerically, both orbit structures II.ii $[6,6,15]$ and
  II.iii $[3,3,6,6,9]$ can be split into $[3,3,3,6,6,6]$. These can be
  distinguished as follows: We are in the first case if and only if one of our
  three orbits of size $6$ consists of pairwise skew lines.

  For $d=2$ and $d=1$, it~seems impossible to give a similar strategy because
  there are $1071$ resp.~$13975$ different orbit structures on the
  $(-1)$-curves.  Furthermore,~given a concrete del Pezzo surface of degree
  $1$ or~$2$, it might be a delicate problem to determine the corresponding
  $\varrho$-maximal subgroup as in~\ref{maximal}.

\item   Several of the values have been known~before. For
  $\rk\Pic(S)=1$, the value $\alpha(S)=1$ is almost immediate from
  E.\,Peyre's~definition. For split del Pezzo surfaces, our
  calculations confirm the values of from~\cite[Theorem~4]{De}.

  In degree $4$, case V confirms the value in \cite[Section~10]{dlBB}.
  In degree $3$, II.i and II.iv are shown
  in~\cite[Remarks~VI.5.9]{J}. II.ii is~\cite[Example~5.6]{J}. Further,~particular
  cases of II.iii, III.iii, and~IV\!.ii appear
  in~\cite[Proposition~5.1]{PT}.
\end{abc}
\end{remarks}

\bibliographystyle{amsplain}

\end{document}